\documentclass[12pt]{amsart}
\usepackage{amsmath,amssymb,amsbsy,amsfonts,amsthm,latexsym,mathabx,
            amsopn,amstext,amsxtra,euscript,amscd,stmaryrd,mathrsfs,
            cite,array,mathtools,enumerate}
\usepackage{url}
\usepackage[colorlinks,linkcolor=blue,anchorcolor=blue,citecolor=blue]{hyperref}

\usepackage{color}
\usepackage{float}

\hypersetup{breaklinks=true}

\AtBeginDocument{%
   \def\MR#1{}
}

\begin{document}

\newtheorem{theorem}{Theorem}
\newtheorem{lemma}[theorem]{Lemma}
\newtheorem{claim}[theorem]{Claim}
\newtheorem{cor}[theorem]{Corollary}
\newtheorem{proposition}[theorem]{Proposition}
\newtheorem{definition}{Definition}
\newtheorem{question}[theorem]{Question}
\newtheorem{remark}[theorem]{Remark}
\newcommand{\hh}{{{\mathrm h}}}

\numberwithin{equation}{section}
\numberwithin{theorem}{section}
\numberwithin{table}{section}

\def\sssum{\mathop{\sum\!\sum\!\sum}}
\def\ssum{\mathop{\sum\ldots \sum}}
\def\dsum{\mathop{\sum \sum}}
\def\iint{\mathop{\int\ldots \int}}

\def\squareforqed{\hbox{\rlap{$\sqcap$}$\sqcup$}}
\def\qed{\ifmmode\squareforqed\else{\unskip\nobreak\hfil
\penalty50\hskip1em\null\nobreak\hfil\squareforqed
\parfillskip=0pt\finalhyphendemerits=0\endgraf}\fi}

\newfont{\teneufm}{eufm10}
\newfont{\seveneufm}{eufm7}
\newfont{\fiveeufm}{eufm5}
%
%
\newfam\eufmfam
     \textfont\eufmfam=\teneufm
\scriptfont\eufmfam=\seveneufm
     \scriptscriptfont\eufmfam=\fiveeufm
%
%
\def\frak#1{{\fam\eufmfam\relax#1}}

\newcommand{\bflambda}{{\boldsymbol{\lambda}}}
\newcommand{\bfmu}{{\boldsymbol{\mu}}}
\newcommand{\bfxi}{{\boldsymbol{\xi}}}
\newcommand{\bfrho}{{\boldsymbol{\rho}}}

\def\fK{\mathfrak K}
\def\fT{\mathfrak{T}}

\def\fA{{\mathfrak A}}
\def\fB{{\mathfrak B}}
\def\fC{{\mathfrak C}}

\def\E{\mathsf {E}}

\def \balpha{\bm{\alpha}}
\def \bbeta{\bm{\beta}}
\def \bgamma{\bm{\gamma}}
\def \blambda{\bm{\lambda}}
\def \bchi{\bm{\chi}}
\def \bphi{\bm{\varphi}}
\def \bpsi{\bm{\psi}}

\def\eqref#1{(\ref{#1})}

\def\vec#1{\mathbf{#1}}


\def\cA{{\mathcal A}}
\def\cB{{\mathcal B}}
\def\cC{{\mathcal C}}
\def\cD{{\mathcal D}}
\def\cE{{\mathcal E}}
\def\cF{{\mathcal F}}
\def\cG{{\mathcal G}}
\def\cH{{\mathcal H}}
\def\cI{{\mathcal I}}
\def\cJ{{\mathcal J}}
\def\cK{{\mathcal K}}
\def\cL{{\mathcal L}}
\def\cM{{\mathcal M}}
\def\cN{{\mathcal N}}
\def\cO{{\mathcal O}}
\def\cP{{\mathcal P}}
\def\cQ{{\mathcal Q}}
\def\cR{{\mathcal R}}
\def\cS{{\mathcal S}}
\def\cT{{\mathcal T}}
\def\cU{{\mathcal U}}
\def\cV{{\mathcal V}}
\def\cW{{\mathcal W}}
\def\cX{{\mathcal X}}
\def\cY{{\mathcal Y}}
\def\cZ{{\mathcal Z}}
\newcommand{\rmod}[1]{\: \mbox{mod} \: #1}

\def\cg{{\mathcal g}}

\def\e{{\mathbf{\,e}}}
\def\ep{{\mathbf{\,e}}_p}
\def\eq{{\mathbf{\,e}}_q}

\def\Tr{{\mathrm{Tr}}}
\def\Nm{{\mathrm{Nm}}}

\def\rE{{\mathrm{E}}}
\def\rT{{\mathrm{T}}}

 \def\SS{{\mathbf{S}}}

\def\lcm{{\mathrm{lcm}}}

\def\t{\tilde}
\def\ov{\overline}

\def\({\left(}
\def\){\right)}
\def\l|{\left|}
\def\r|{\right|}
\def\fl#1{\left\lfloor#1\right\rfloor}
\def\rf#1{\left\lceil#1\right\rceil}
\def\flq#1{\langle #1 \rangle_q}

\def\mand{\qquad \mbox{and} \qquad}

\newcommand{\commIg}[1]{\marginpar{%
\begin{color}{magenta}
\vskip-\baselineskip 
\raggedright\footnotesize
\itshape\hrule \smallskip Ig: #1\par\smallskip\hrule\end{color}}}

\newcommand{\commSi}[1]{\marginpar{%
\begin{color}{blue}
\vskip-\baselineskip 
\raggedright\footnotesize
\itshape\hrule \smallskip Si: #1\par\smallskip\hrule\end{color}}}




\hyphenation{re-pub-lished}

\mathsurround=1pt

\def\bfdefault{b}
\overfullrule=5pt

\def \F{{\mathbb F}}
\def \K{{\mathbb K}}
\def \Z{{\mathbb Z}}
\def \Q{{\mathbb Q}}
\def \R{{\mathbb R}}
\def \C{{\\mathbb C}}
\def\Fp{\F_p}
\def \fp{\Fp^*}

\def\Smn{S_{k,\ell,q}(m,n)}

\def\Kmn{\cK_p(m,n)}
\def\psmn{\psi_p(m,n)}

\def\SM{\cS_{k,\ell,q}(\cM)}
\def\SMN{\cS_{k,\ell,q}(\cM,\cN)}
\def\SAMN{\cS_{k,\ell,q}(\cA;\cM,\cN)}
\def\SABMN{\cS_{k,\ell,q}(\cA,\cB;\cM,\cN)}

\def\SIJq{\cS_{k,\ell,q}(\cI,\cJ)}
\def\SAJq{\cS_{k,\ell,q}(\cA;\cJ)}
\def\SABJq{\cS_{k,\ell,q}(\cA, \cB;\cJ)}

\def\sM{\cS_{k,q}^*(\cM)}
\def\sMN{\cS_{k,q}^*(\cM,\cN)}
\def\sAMN{\cS_{k,q}^*(\cA;\cM,\cN)}
\def\sABMN{\cS_{k,q}^*(\cA,\cB;\cM,\cN)}

\def\sIJq{\cS_{k,q}^*(\cI,\cJ)}
\def\sAJq{\cS_{k,q}^*(\cA;\cJ)}
\def\sABJq{\cS_{k,q}^*(\cA, \cB;\cJ)}
\def\sABJp{\cS_{k,p}^*(\cA, \cB;\cJ)}

 \def \xbar{\overline x}

\author[S.  Macourt] {Simon Macourt}
\address{Department of Pure Mathematics, University of New South Wales,
Sydney, NSW 2052, Australia}
\email{s.macourt@student.unsw.edu.au}

\title{Decomposition of subsets of finite fields}

\begin{abstract} 
We extend a bound of Roche-Newton, Shparlinski and Winterhof which says any subset of a finite field can be decomposed into two disjoint subset $\cU$ and $\cV$ of which the additive energy of $\cU$ and $f(\cV)$ are small, for suitably chosen rational functions $f$. We extend the result by proving equivalent results over multiplicative energy and the additive and multiplicative energy hybrids. 
\end{abstract}
\keywords{}
\subjclass[2010]{}

\maketitle

\section{Introduction}
\subsection{Background}
Let $\F_q$ denote the finite field of $q$ elements of characteristic $p$.

Given two sets $\cU, \cV \subset \F_q$ we define their sum and product sets as
\begin{align*}
\cU + \cV = \{u+v : u\in \cU, v\in \cV\} \quad \text{and} \quad \cU \cdot \cV = \{uv : u\in \cU, v\in \cV\}.
\end{align*}
We define the additive and multiplicative energy of a set as follows
\begin{align*}
E^+(\cU)&= \#\{(u_1,u_2,u_3,u_4) \in \cU^4 : u_1+u_2=u_3+u_4\} \\
E^\times(\cU) &=\#\{(u_1,u_2,u_3,u_4) \in \cU^4 : u_1u_2=u_3u_4\}.
\end{align*}
We mention the sum-product problem which suggests that at least one of the sets $\cU + \cV$ and $\cU \cdot \cV$ must be large. This problem has been studied extensively in recent years, see \cite{BKT}. There is a natural relation to the sum-product problem to bounds on additive and multiplicative energy. For example, by applying the Cauchy-Schwarz inequality one can see that
\begin{align*}
E^\times(\cU) \ge \frac{|\cU|^4}{|\cU\cdot \cU|},
\end{align*}
and similarly for additive results. It follows that strong upper bounds on energy results correspond to strong lower bounds on the relevant sum-product estimate and vice-versa.

Balog and Wooley \cite{BalWoo} proved that in finite fields the set $\cU$ can be decomposed into a disjoint union of subsets $\cV$ and $\cW$ such that $E^+(\cV)$ and $E^\times(\cW)$ are both small. These results have been improved on by Konyagin and Shkredov \cite{KonShk} and Rudnev, Shkredov and Stevens \cite{RSS}.

Our main results are an extension of \cite{R-NSW}, which themselves are a generalisation of the Balog-Wooley decomposition \cite[Theorem 1.3]{BalWoo}. Here we extend the results of \cite{R-NSW} to multiplicative energy and the hybrid cases of both additive and multiplicative energy.

\subsection{Notation}
For $a \in \F_q$ and a rational function $f \in \F_q(X)$ we use $r^+_{\cU,\cV}(f,a)$ to denote the number of solutions to $f(u) + f(v)=a, (u,v) \in \cU\times \cV$. Similarly we use $r^\times_{\cU,\cV}(f,a)$ to denote the number of solutions to $f(u)f(v)=a$. If $\cU=\cV$ we write $r^+_{\cU}(f,a)$ and if $f(X)=X$ we write $r^+_{\cU,\cV}(a)$.

For this paper we use the convention that capital letters in italics, such as $\cU$, will be used to represent sets. Corresponding capital letters in roman will denote their cardinalities, such as $U=|\cU|$. We also use $\cX$ and $\Psi^*$ to denote the sets of additive and multiplicative characters respectively, and we will use the lower case $\chi$ and $\psi$ to represent their respective characters.

Throughout the paper we use the notation $A \ll B$ to indicate \linebreak[4] $|A| \le c|B|$ for some absolute constant $c$. We also use the notation $A \ll_k B$ for when the constant $c$ depends on some parameter $k$. We also equivalently write $A=O(B)$ and $A=O_k(B)$.

\subsection{Main Results}
Here we extend the result of \cite[Theorem 1.1]{R-NSW} to multiplicative energy and a hybrid of additive and multiplicative energies.
\begin{theorem}\label{thm:E(S)Ef(T)}
For any set $\cA \subset \F^*_q$ and any rational function $f \in \F_q(X)$ of degree $k$ which is not of the form $f(X)=rg(X)^dX^\lambda$ where $d | q-1$, there exist disjoint sets $\cS, \cT \subset \cA$ such that $\cA=\cS\cup\cT$ and 
\begin{align*}
\max\{E^\times(\cS), E^\times(f(\cT))\} \ll_k \frac{A^3}{M(A)}
\end{align*}
where 
\begin{align*}
M(\cA) = \min\left\{\frac{q^{1/2}}{A^{1/2}(\log A)^{11/4}}, \frac{A^{4/5}}{q^{2/5}(\log A)^{31/10}}\right\}.
\end{align*}
\end{theorem}

\begin{theorem}\label{thm:ExE+}
For any set $\cA \subset \F^*_q$ and any rational function $f \in \F_q(X)$ of degree $k$ which is not of the form $f(X)=g(X^p)-g(X)+\lambda X+\mu$, there exist disjoint sets $\cS, \cT \subset \cA$ such that $\cA=\cS\cup\cT$ and 
\begin{align*}
\max\{E^\times(\cS), E^+(f(\cT))\} \ll_k \frac{A^3}{M(A)}.
\end{align*}
\end{theorem}
\begin{theorem} \label{thm:E+Ex}
For any set $\cA \subset \F^*_q$ and any rational function $f \in \F_q(X)$ of degree $k$ which is not of the form $f(X)=rg(X)^dX^\lambda$ where $d | q-1$, there exist disjoint sets $\cS, \cT \subset \cA$ such that $\cA=\cS\cup\cT$ and 
\begin{align*}
\max\{E^+(\cS), E^\times(f(\cT))\} \ll_k \frac{A^3}{M(A)}.
\end{align*}
\end{theorem}

\section{Sum-Product}
\subsection{Preliminary Results}
We give a series of lemmas, the proofs of which follow those of \cite{R-NSW} with multiplicative characters replacing additive characters and other equivalent substitutions.
\begin{lemma} \label{lem:Chartimes} 
Let $(\chi, \psi) \in \cX \times \Psi^*$. For any rational function $f \in \F_q(X)$ of degree $k$, if $\chi$ is non-trivial, and not of the form $f(X)=rg(X)^dX^{\lambda}$ where $d$ is the order of $\chi$, if $\psi$ is trivial, we have
\begin{align*}
\sum_{u\in \cU} \sum_{v \in \cV}\chi(f(uv))\psi(uv)\ll_k \sqrt{UVq}.
\end{align*}
\end{lemma}
\begin{proof}
Let
\begin{align*}
\Sigma = \sum_{u \in \cU} \sum_{v \in \cV}\chi(f(uv))\psi(uv).
\end{align*}
Then,
\begin{align*}
\Sigma &= \sum_{x \in \F_q} \psi(x)\chi(f(x))\frac{1}{q-1}\sum_{\lambda =1}^{q-1} \sum_{u\in \cU} \sum_{v \in \cV} \chi((uvx^{-1})^\lambda) \\
& = \frac{1}{q-1} \sum_{\lambda=1}^{q-1}\sum_{x \in \F_q} \psi(x)\chi(f(x)(x^{-1})^\lambda) \sum_{u\in \cU} \chi(u^\lambda) \sum_{v \in \cV} \chi(v^\lambda).
\end{align*}
By the Weil bound we have
\begin{align*}
\Sigma \ll q^{-1/2}\sum_{\lambda =1}^{q-1} \left| \sum_{u\in \cU} \chi(u^\lambda)\right| \left|\sum_{v \in \cV} \chi(v^\lambda)\right|.
\end{align*}
Using the Cauchy-Schwarz inequality we obtain
\begin{align*}
\sum_{\lambda \in \F_q} \left| \sum_{u\in \cU} \chi(u^\lambda)\right| &\left|\sum_{v \in \cV} \chi(v^\lambda)\right| \\
&\le \left(\sum_{\lambda \in \F_q} \left| \sum_{u\in \cU} \chi(u^\lambda)\right|^2 \right)^{1/2} \left(\sum_{\lambda \in \F_q}  \left|\sum_{v \in \cV} \chi(v^\lambda)\right|^2 \right)^{1/2} \\
&\le (q^2UV)^{1/2}.
\end{align*}
\end{proof}

\begin{lemma} \label{lem:Jtimes}
Suppose $\cU, \cV, \cY, \cZ \subset \F_q$. For any rational function $f \in \F_q(X)$ of degree $k$ which is not of the form $f(X)=rg(X)^dX^{\lambda}$ where $d|q-1$ and $d\ge 2$, the number of solutions $J$ to the equation
\begin{align*}
f(uv)=yz \qquad (u,v,y,z)\in \cU\times\cV\times\cY\times\cZ
\end{align*}
satisfies the bound
\begin{align*}
J \le \frac{UVYZ}{q-1} + O_k((UVYZq)^{1/2}).
\end{align*}
\end{lemma}

\begin{proof}
Using the approximate orthogonality of multiplicative characters, we have
\begin{align*}
J \le \sum_{(u,v,y,z)\in \cU \times \cV \times \cY \times \cZ} \frac{1}{q-1} \sum_{\chi \in \cX} \chi(f(uv)(yz)^{-1}).
\end{align*}
Re-arranging and separating the contribution from the trivial character 
\begin{align*}
J - \frac{UVYZ}{q-1} \le \frac{1}{q-1} \sum_{\chi \in \cX^*} \left| \sum_{(u,v)\in \cU \times \cV} \chi(f(uv)) \right| \left| \sum_{y \in \cY} \chi(y^{-1})\right| \left| \sum_{z\in \cZ}\chi(z^{-1}) \right|.
\end{align*}
Now by Lemma \ref{lem:Chartimes} with the trivial additive character, we have
\begin{align*}
J - \frac{UVYZ}{q-1} &\ll_k \frac{\sqrt{UVq}}{q-1}  \sum_{\chi \in \cX^*} \left| \sum_{y \in \cY} \chi(y^{-1})\right| \left| \sum_{z\in \cZ}\chi(z^{-1}) \right| \\
& \ll_k \frac{\sqrt{UV}}{q^{1/2}} \cdot (q^2YZ)^{1/2}.
\end{align*}
This completes the proof.
\end{proof}

\begin{lemma} \label{lem:rtimes}
Let $\cA, \cS, \cU \subset \F^*_q$. Let $u >0$ be such that $r^\times_{\cS,\cA^{-1}}(x) \ge u$ for all $x \in \cU$. Let $k$ be a fixed positive integer and suppose also that
\begin{align*}
\tau \ge 2 \frac{kASU}{uq}.
\end{align*}
Then, for any rational function $f \in \F_q(X)$ of degree $k$ which is not of the form $f(X)=rg(X)^dX^{\lambda}$ where $d|p-1$ and $d\ge2$, we have
\begin{align*}
\# \{x \in \F_q : r^\times_{\cU}(f,x) \ge \tau \} \ll_k \frac{AUSq}{u^2\tau^2}.
\end{align*}
\end{lemma}
\begin{proof}
Our proof follows \cite[Lemma 2.3]{R-NSW} where here we replace $r_\cU(f,x)$ with $r^\times_\cU(f,x)$. Define
\begin{align*}
\cR = \{x \in \F_q : r^\times_\cU(f,x)\ge \tau\}.
\end{align*}
Clearly,
\begin{align*}
\tau R \le \sum_{x \in \cR} r_\cU(f,x) = \{(x,y,z)\in \cR\times \cU \times \cU: x=f(y)f(z)\}.
\end{align*}
Now $r_{\cS,\cA^{-1}}(z) \ge u$ for $z\in \cU$, hence
\begin{align*}
&\# \{(x,y,z)\in \cR\times \cU \times \cU:x=f(y)+f(z)\}\\
&\qquad \qquad \le u^{-1}\# \{(v,w,x,y)\in \cS\times \cA\times \cR \times \cU:x=f(y)f(vw^{-1})\}.
\end{align*}
Therefore, we have
\begin{align*}
\tau U R &\le \# \{(v,w,x,y)\in \cS\times \cA\times \cR \times \cU:x=f(y)f(vw^{-1})\}\\
&\le k\cdot  \# \{(v,w,x,z)\in \cS\times \cA\times \cR \times f(\cU):x=zf(vw^{-1})\}.
\end{align*}
We then apply Lemma \ref{lem:Jtimes} to obtain
\begin{align*}
\tau U R  \le \frac{kARSU}{q} + O_k((ARSUq)^{1/2}).
\end{align*}
The assumed lower bound on $\tau$ implies
\begin{align*}
\tau U R  \ll_k (ARSUq)^{1/2}.
\end{align*}
This concludes the proof.
\end{proof}
\begin{lemma} \label{lem:Etimes}
Let $\cA_1, \dots, \cA_n \subset \F^*_q$. Then
\begin{align*}
E^\times \(\bigcup_{i=1}^n \cA_i\) \le \(\sum_{n=1}^n E^\times (\cA_i)^{1/4} \) ^4.
\end{align*}
\end{lemma}
\begin{proof}
We assume the sets $\cA_i, \dots, \cA_n$ are disjoint. Then using the Cauchy-Schwarz inequality twice we have,
\begin{align*}
E^\times \left(\bigcup_{i=1}^n \cA_i \right) &= \sum_{i,j,k,\ell =1}^{n} \sum_{x \in \F_q} r^\times_{\cA_i,\cA_j}(x)r^\times_{\cA_k,\cA_\ell}(x) \\
&\le \sum_{i,j,k,\ell =1}^{n} \left(\sum_{x \in \F_q} r^\times_{\cA_i,\cA_j}(x)^2\right)^{1/2}\left(\sum_{x \in \F_q} r^\times_{\cA_k,\cA_\ell}(x)^2\right)^{1/2}\\
&= \left(\sum_{i,j =1}^{n} \left(\sum_{x \in \F_q} r^\times_{\cA_i,\cA_j}(x)^2\right)^{1/2} \right)^2\\
&=\left(\sum_{i,j =1}^{n} \left(\sum_{x \in \F_q} r^\times_{\cA_i,\cA_i^{-1}}(x) r^\times_{\cA_j,\cA_j^{-1}}(x)\right)^{1/2} \right)^2\\
&\le \left(\sum_{i,j =1}^{n} \left(\sum_{x \in \F_q} r^\times_{\cA_i,\cA_i^{-1}}(x)^2\right)^{1/4} \left( \sum_{x \in \F_q} r^\times_{\cA_j,\cA_j^{-1}}(x)^2\right)^{1/4} \right)^2\\
&= \left(\sum_{i =1}^{n} \left(\sum_{x \in \F_q} r^\times_{\cA_i,\cA_i^{-1}}(x)^2\right)^{1/4} \right)^4 = \left(\sum_{i=1}^{n} E^\times(\cA_i)^{1/4} \right)^4.
\end{align*}
This concludes the proof.
\end{proof}
%
%
%
%
%
%

\begin{lemma} \label{lem:UEfU}
Let $\cA \subset \F_q$. Then for any rational  function $f \in \F_q(X)$ of degree $k$ which is not of the form $f(X)=rg(X)^dX^{\lambda}$ where $d|p-1$ and $d\ge2$, there exists $\cU \subset \cA$ of cardinality $U$ such that 
\begin{align*}
U \gg \frac{E^\times(\cA)^{1/2}}{A^{1/2}(\log A)^{7/4}}
\end{align*}
and
\begin{align*}
E^\times(f(\cU)) \ll_k \frac{AU^6q^{-1}(\log A)^{11/2}+AU^3q(\log A)^6}{E^\times(\cA)}.
\end{align*}
\end{lemma}

\begin{proof}
Clearly,
\begin{align*}
E^\times (\cA) = \sum_{x \in \cA\cdot \cA}r^\times_{\cA}(x)^2.
\end{align*}
We dyadically decompose this sum and define the set
\begin{align*}
\cS^\times = \{x\in \cA \cdot \cA : \rho \le r^\times_\cA(x) < 2\rho \}
\end{align*}
with some integer $1 \le \rho \le A$ where $\rho$ is a power of 2, and such that
\begin{align} \label{eq:rho}
\rho^2 S \gg \frac{E^\times(\cA)}{\log A}.
\end{align}
Consider
\begin{align*}
\cP = \{(a,b) \in \cA \times \cA: ab\in \cS^\times \}.
\end{align*}
Now we have
\begin{align} \label{eq:rhoPS}
\rho S\le P < 2\rho S.
\end{align}
We then make another dyadic decomposition of $\cS$ to find a large subset supported on vertical lines. That is, we define
\begin{align*}
\cA_x = \{ y :(x,y) \in \cP\}.
\end{align*}
Therefore, for some $s$ there exists a dyadic set
\begin{align*}
\cV = \{x \in \cA:s\le \cA_x<2s\}
\end{align*}
such that
\begin{align}\label{eq:Vs}
Vs \gg \frac{P}{\log A}\gg \frac{\rho S}{\log A}.
\end{align}
We now separate into two cases. First, suppose 
\begin{align*}
V \ge \frac{s}{(\log A)^{1/2}}.
\end{align*}
Then for any $x \in \cV$, there exist
\begin{align*}
y_1, y_2, \dots, y_s \in \cA_x \subset \cA
\end{align*}
such that $(x,y_i) \in \cP$ for all $1 \le i \le s$. Therefore
\begin{align*}
xy_1,xy_2, \dots, xy_s \in \cS^\times.
\end{align*}
It follows that $r^\times_{\cS^\times, \cA^{-1}}(x) \ge s$ for every $x \in \cV$ and in this case we define
\begin{align} \label{eq:U=Vu=s}
\cU = \cV \qquad \text{and} \qquad u=s.
\end{align}
Now suppose
\begin{align*}
V < \frac{s}{(\log A)^{1/2}}.
\end{align*}
We now consider the point set
\begin{align*}
\cQ = \{ (x,y) \in \cP : x \in \cV\}.
\end{align*}
As before, for any $x \in \cV$ there exist at least $s$ values of $y \in \cA_x \subset \cA$ with $(x,y) \in \cP$. Hence $Q \ge Vs$.

For any $y \in \F_q$ we define 
\begin{align*}
\cB_y = \{x : (x,y) \in \cQ.
\end{align*}
Clearly,
\begin{align*}
\sum_{y \in \cA}B_y = Q.
\end{align*}
Therefore, for some $t$ there exists a dyadic set
\begin{align*}
\cW = \{y \in \cA : t \le B_y < 2t\}
\end{align*}
such that
\begin{align} \label{eq:Wt}
Wt \gg \frac{Q}{\log A} \ge \frac{Vs}{\log A}.
\end{align}
Now since $\cQ \subset \cV \times \cA$ we also have $t \le V$. From \eqref{eq:Wt} and our assumption on $s$ we have
\begin{align*}
WV \ge Wt \gg \frac{Vs}{\log A} > \frac{V^2}{(\log A)^{1/2}}
\end{align*}
hence
\begin{align} \label{eq:W}
W \gg \frac{V}{(\log A)^{1/2}} \ge \frac{t}{(\log A)^{1/2}}.
\end{align}
Now, by \eqref{eq:Wt} and \eqref{eq:Vs}
\begin{align} \label{eq:Wtrho}
Wt \gg  \frac{Vs}{\log A} \gg  \frac{\rho S}{(\log A)^2}.
\end{align}
Now, let $y \in \cW$. Then there exist $x_1, \dots, x_t \in \cA$ such that $(x_i,y)\in \cP$ for all $1 \le i \le t$. Therefore,
\begin{align*}
x_1y, \dots, x_ty \in \cS.
\end{align*}
Then $r^\times_{\cS,\cA^{-1}}(y)\ge t$ for every $y \in \cW$.

We then take 
\begin{align} \label{eq:cUcW}
\cU =\cW \qquad \text{and} \qquad u=t.
\end{align}
It is clear for both \eqref{eq:U=Vu=s} and \eqref{eq:cUcW} we have $\cU \subset \cA$, \begin{align} \label{eq:U>}
U \gg \frac{u}{(\log A)^{1/2}}
\end{align}
and
\begin{align} \label{eq:uU>}
uU \gg \frac{\rho S}{(\log A)^2}
\end{align}
where $r^\times_{\cS,\cA^{-1}}(x)\ge u$ for all $x \in \cU$. Multiplying \eqref{eq:U>} and \eqref{eq:uU>} and using \eqref{eq:rho} we obtain
\begin{align} \label{eq:U2>}
U^2 \gg \frac{\rho S}{(\log A)^{5/2}} \gg \frac{E^\times(\cA)}{A (\log A)^{7/2}}.
\end{align}
We now need a bound on $E^\times (f(\cU))$. We have
\begin{align} \label{eq:EfU}
E^\times (f(\cU)) = \sum_{x \in \F_q} r_{f(\cU)}^\times(x)^2 \le \sum_{x \in \F_q}r_\cU^\times(f,x)^2.
\end{align}
We define the set
\begin{align*}
\cR_0= \left\{x \in \F_q : r_\cU^\times (f,x) \le 2\frac{kASU}{uq}\right\}
\end{align*}
and for $J=\lceil \log A/ \log2 \rceil$, we define the sets
\begin{align*}
\cR_j = \left\{ x \in \F_q : 2^j\frac{kASU}{uq} < r_\cU^\times(f,x) \le 2^{j+1}\frac{AkSU}{uq} \right\}, \ j=1,\dots,J.
\end{align*}
Since,
\begin{align*}
 \sum_{x \in \F_q} r_\cU^\times(f,x) = U^2
\end{align*}
we have
\begin{align}\label{eq:R0}
\sum_{x \in \R_0}r_\cU^\times(f,x)^2 \le 2\frac{kASU}{uq}  \sum_{x \in \F_q}r_\cU^\times(f,x) \ll \frac{kASU^3}{uq}.
\end{align}
For $i=1, \dots, J$, we apply Lemma \ref{lem:rtimes} with 
\begin{align*}
\tau = 2^j\frac{AkSU}{uq}
\end{align*}
to obtain
\begin{align} \label{eq:Rj}
\sum_{x \in \R_j}r_\cU^\times(f,x)^2 \le (2\tau)^2 R_j \ll_k \frac{ASUq}{u^2}.
\end{align}
Combining \eqref{eq:R0} and \eqref{eq:Rj} we get
\begin{align} \label{eq:Etimes2}
E^\times (f(\cU)) \ll_k  \frac{ASU^3}{uq}+ \frac{ASUq}{u^2}\log A.
\end{align}

Now, multiplying \eqref{eq:uU>} with \eqref{eq:U2>} and applying \eqref{eq:rho}, we obtain
\begin{align*}
uU^3 \gg \frac{\rho^2S^2}{(\log A)^{9/2}} \gg \frac{SE^\times(\cA)}{(\log A)^{11/2}}
\end{align*}
which gives
\begin{align} \label{eq:S/u}
\frac{S}{u} \ll \frac{U^3 (\log A)^{11/2}}{E^\times(\cA)}.
\end{align}
Also, squaring \eqref{eq:uU>} and applying \eqref{eq:rho}
\begin{align*}
u^2U^2 \gg \frac{\rho^2S^2}{(\log A)^4} \gg \frac{SE^\times(\cA)}{(\log A)^5}
\end{align*}
which gives
\begin{align}\label{eq:S/u2}
\frac{S}{u^2} \ll \frac{U^2(\log A)^5}{E^\times(\cA)}.
\end{align}
Applying \eqref{eq:S/u} and \eqref{eq:S/u2} into the first and second terms of \eqref{eq:Etimes2} respectively we obtain
\begin{align*}
E^\times(f(\cU)) \ll_k \frac{AU^6q^{-1}(\log A)^{11/2}+AU^3q(\log A)^6}{E^\times(\cA)}.
\end{align*}
This concludes the proof.
\end{proof}

\begin{cor}\label{cor:ExE+}
Let $\cA \subset \F_q$. Then for any rational  function $f \in \F_q(X)$ of degree $k$ which is not of the form $f(X)=g(X^p)-g(X)+\lambda X+\mu$, there exists $\cU \subset \cA$ of cardinality $U$ such that 
\begin{align*}
U \gg \frac{E^\times(\cA)^{1/2}}{A^{1/2}(\log A)^{7/4}}
\end{align*}
and
\begin{align*}
E^+(f(\cU)) \ll_k \frac{AU^6q^{-1}(\log A)^{11/2}+AU^3q(\log A)^6}{E^\times(\cA)}.
\end{align*}
\end{cor}
\begin{proof}
We follow the proof of Lemma \ref{lem:UEfU}, however we replace $E^\times$ with $E^+$ in \eqref{eq:EfU}, and then use analogous results following from \cite[Equation 2.12]{R-NSW}.
\end{proof}

\begin{cor}\label{cor:E+Ex}
Let $\cA \subset \F_q$. Then for any rational  function $f \in \F_q(X)$ of degree $k$ which is not of the form $f(X)=rg(x)^dx^{\lambda}$ where $d|p-1$ and $d\ge2$, there exists $\cU \subset \cA$ of cardinality $U$ such that 
\begin{align*}
U \gg \frac{E^+(\cA)^{1/2}}{A^{1/2}(\log A)^{7/4}}
\end{align*}
and
\begin{align*}
E^\times(f(\cU)) \ll_k \frac{AU^6q^{-1}(\log A)^{11/2}+AU^3q(\log A)^6}{E^+(\cA)}.
\end{align*}
\end{cor}
\begin{proof}
We follow the proof of \cite[Lemma 2.5]{R-NSW}, however we replace $E^+$ with $E^\times$ in equation (2.12) and the proceed as in our Lemma \ref{lem:UEfU}.
\end{proof}

\subsection{Proofs of Theorems \ref{thm:E(S)Ef(T)}, \ref{thm:ExE+} and \ref{thm:E+Ex}}
\begin{proof}
The proofs follow that of \cite[Theorem 1.1]{R-NSW}, but for Theorem \ref{thm:E(S)Ef(T)}  our new multiplicative results from Lemma \ref{lem:UEfU} are used in place of additive results. For Theorems \ref{thm:ExE+} and \ref{thm:E+Ex} our mixed results from Corollaries \ref{cor:ExE+} and \ref{cor:E+Ex} are used respectively.
\end{proof}

\end{document}